\documentclass[11pt]{amsart}

\newtheorem{theorem}{Theorem}[section]
\newtheorem{thrm}{Theorem}[section]
\newtheorem{lmm}[theorem]{Lemma}
\newtheorem{crllr}[theorem]{Corollary}
\newtheorem{prpstn}[theorem]{Proposition}

\theoremstyle{definition}

\theoremstyle{remark}

\newcommand{\R}{{\mathbb R}}

\newcommand{\cC}{\mathcal C}

\begin{document}
\title[Speed of Spread for Fractional Diffusion Equations]
{On the Speed of Spread for Fractional Reaction-Diffusion Equations}

\author{Hans Engler}
\address{Dept. of Mathematics, Georgetown University\\
Box 571233\\ Washington, DC 20057\\
USA}
\email{engler@georgetown.edu}

\begin{abstract}
The fractional reaction diffusion equation $\partial_tu + Au = g(u)$ is discussed, where $A$ is a fractional differential operator on $\R$ of order $\alpha \in (0,2)$, the $C^1$ function $g$ vanishes at $\zeta = 0$ and $\zeta = 1$ and either $g \ge 0$ on $(0,1)$ or $g < 0$ near $\zeta = 0$. In the case of non-negative $g$, it is shown that solutions with initial support on the positive half axis spread into the left half axis with unbounded speed if  $g(\zeta)$ satisfies some weak growth condition near $\zeta = 0$ in the case $\alpha > 1$, or if $g$ is  merely positive on a sufficiently large interval near $\zeta = 1$ in the case $\alpha < 1$. On the other hand, it shown that solutions spread with finite speed if $g'(0) < 0$. The  proofs use comparison arguments and a new family of travelling wave solutions for this class of problems. 

\end{abstract}

\maketitle
\noindent

{\footnotesize }


\date{July 2009}

\section{Introduction}

The scalar reaction-diffusion equation
\begin{equation} \partial_t u(x,t) - \partial_x^2 u(x,t) = g(u(x,t)) \label{rde2}
\end{equation}
has been the subject of long study, beginning with the celebrated paper \cite{kpp}. The authors of \cite{kpp} proposed this equation, with $g$ positive and concave on $(0,1)$ such that $g(0) = g(1) = 0$, as a model for a population that undergoes logistic growth and Brownian diffusion. If $u(x,0) = H(x)$, the Heaviside function, and $g(u) = u -u^2$, the equation in fact has an exact probabilistic interpretation, see \cite{mckean}. For this problem, is known that solutions approach a wave profile $\psi$ in the sense that
\[u(x+m(t),t) \to \psi(x) \quad (t \to \infty)
\]
where $m(t)$ is the median, $u(m(t),t) = \frac12$. It turns out that $m(t) = c^*t + O(\log t)$ for a suitable asymptotic finite wave speed $c^*$. Larger asymptotic speeds are only possible if the initial data are supported on $\R$. A more general result, given in \cite{aronwein}, implies that there is a critical speed $c^*$ such that for fairly general initial data $u(\cdot,0)$ that are non-negative and supported on $(0,\infty)$,
\begin{equation} \lim_{t \to \infty} \limsup_{x<-ct} u(x,t) =0  \label{est_0}
\end{equation}
whenever $c>c^*$ and
\begin{equation} \lim_{t \to \infty} \liminf_{x>-ct} u(x,t) =1  \label{est_1}
\end{equation}
whenever $c<c^*$.  If $u$ denotes a quantity that is to be avoided and whose spread is governed by \eqref{rde2}, a runner may escape from it by running to $-\infty$ at a speed $c>c^*$, but this quantity will catch up with and engulf her if her speed is $c<c^*$.

\smallskip
Equation \eqref{rde2} was also derived in \cite{allencahn} to describe antiphase domain coarsening in alloys. In this situation, $g(0)=g(1)=g(u^*) = 0$ for some $u^* \in (0,1)$, and $g<0$ on $(0,u^*)$, $g>0$ on $(u^*,1)$. In this case there exists exactly one wave speed $c^*$ with associated wave profile. In particular, \eqref{est_0} still holds for this $c^*$. The first of the two cases (the KPP case) corresponds to "pulled" fronts (the state $u=0$ is unstable) while the second case (the Allen-Cahn case) results in a "pushed" front (the state $u=0$ is stable). More on these two fundamentally different situations may be found in \cite{ebert1} and the references given there. A vast range of applications leading to related models is discussed in \cite{fortpujol}.

\smallskip
The purpose of this note is a study of the fractional reaction-diffusion equation
\begin{equation} \partial_t u(x,t) + A u(x,t) = g(u(x,t)) \, . \label{rde_alpha}
\end{equation}
Here $A$ is a pseudo-differential operator with symbol $p$ that is homogeneous of degree $\alpha \in (0,2]$, such that $p(-\lambda) = \overline{p(\lambda)}$ and $|p(1)| =1$. Then we can write $p$ in the form
\begin{equation}
p(\lambda) =   e^{- i \frac{\pi}{2} sign(\lambda)\rho}|\lambda|^\alpha
\label{symbol1}
\end{equation}
where $\rho \in \R$. There will be additional restrictions on the parameter $\rho$ in section 2. For $ \rho = 0$ we obtain fractional powers of the usual negative one-dimensional Laplacian, abbreviated often by $(-\Delta)^{\alpha/2}$. There are various real variable representations of such operators, e.g. as a singular integral operator or as a limit of suitable difference operators; see \cite{bkm2} where this is explained in more detail. The function $g$ is always assumed to satisfy $g(0) = g(1) = 0$. We are interested in both the KPP-case, i.e. $g(\zeta) \ge 0$ for $0 < \zeta < 1$, and the Allen-Cahn case, i.e. $g(\zeta) < 0$ for $\zeta$ near  $0$ and $g(\zeta) > 0$ for $\zeta$ near $1$.

\smallskip

This class of equations has recenty been proposed as a model for reaction and anomalous diffusion; see \cite{bkm2,castil,mvv1, mvv2, zanette,zas}. It should be noted that the term "anomalous diffusion" is also used for situations in which the first order time derivative is replaced by a fractional order derivative. Another generalization of \eqref{rde2} consists in allowing time delays; see \cite{schaaf}. These further generalizations will not be discussed here.

\medskip

There is strong evidence that the equation \eqref{rde_alpha} does not admit traveling wave solutions if $0 < \alpha < 2$ and $g$ is positive and concave on $(0,1)$. Rather, numerical results in \cite{castil} and \cite{bkm3} suggest that for initial data that are supported on the positive half axis and increase there from 0 to 1, the median satisfies $m(t) \sim -e^{ct}$ for some $c>0$. In \cite{cabre1}, the estimates
\begin{equation} \lim_{t \to \infty} \limsup_{x<-e^{ct}} u(x,t) =0,  \quad \lim_{t \to \infty} \liminf_{x>-e^{dt}} u(x,t) =1 \label{est_exp}
\end{equation}
are shown to hold for such initial data whenever $c>c^*>d$, where $c^* = g'(0)/\alpha$. Thus the asymptotic speed of spread grows exponentially. For the case where $g<0$ on some interval $(0,u^*)$, the results in \cite{golovin} and \cite{zanette} suggest on the other hand that there exist wave profile solutions that move with constant speed, although no rigorous proofs are given there.

\medskip
The main results of this note take the form \eqref{est_0} and \eqref{est_1}. It is shown that for a large class of right hand sides $g$ that are non-negative on $(0,1)$, the estimate \eqref{est_1} holds for \emph{all} speeds $c$; that is, the speed of spread is unbounded. It is not necessary to assume that $g'(0) > 0$, and if $\alpha < 1$, $g$ may even be zero near $0$. On the other hand, if $g'(0) < 0$, then it is shown that \eqref{est_0} holds for some finite $c$, that is, there is always a bound on the speed of spread. These results are stated and proven in section 3. This is done by employing comparison arguments together with a set of travelling wave solutions that is constructed in section 2 and that may be of independent interest. Some basic existence and comparison results for \eqref{rde_alpha} are sketched in section 4.

\section{A Class of Travelling Wave Solutions}

In this section, it will be shown that cumulative distribution functions of stable probability distributions with parameters $\alpha, \, \beta$ lead to traveling wave solutions $u(x,t) = U(x+ct)$  of \eqref{rde_alpha}, for suitable functions $g$. A two parameter family will be constructed for each possible choice of $\alpha$ and $\rho$, one parameter being the speed $c$. The main contribution of this section is the characterization of the nonlinear function $g$ that is required to make the equation hold.

\medskip
Consider the probability density function $f_{\alpha \beta}$ of a stable probability distribution with index of stability $\alpha \in (0,2)$, skewness parameter $\beta \in (-1,1)$, scale parameter $\gamma =1$, and location parameter $\delta = 0$, where Zolotarev's parametrization (B) (see \cite{zolo}) is used for $\alpha \ne 1$ and form (C) is used if $\alpha = 1$. The characteristic function (Fourier transform) of $f_{\alpha\beta}$ then is
\begin{equation} \label{char_fun}
\lambda \mapsto e^{-|\lambda|^\alpha \omega(\lambda)}\, ,
\end{equation}
where
\[\omega(\lambda) = \begin{cases} e^{-i \frac{\pi}{2}sign(\lambda) \beta (\alpha - 1 + sign(1-\alpha))}   \quad (\alpha \ne 1) \\
e^{-i \frac{\pi}{2}sign(\lambda) \beta } \quad (\alpha = 1)
\end{cases} \, .
\]

\smallskip
The corresponding cumulative distribution function is denoted by $F_{\alpha \beta}$. It is known that $f_{\alpha \beta}$ is positive, infinitely differentiable, and unimodal. Also,  as $x \to \infty$, there are asymptotic representations
\begin{eqnarray}
1-F_{\alpha \beta}(x) &\sim& \sum_{j\ge 1}c_{j\alpha\beta} x^{-j\alpha}
\\
f_{\alpha \beta}(x) &\sim& \sum_{j\ge 1} \alpha c_{j\alpha\beta} x^{-1-j\alpha}
\label{asymp+}
\end{eqnarray}
and as $x \to - \infty$
\begin{eqnarray}
F_{\alpha \beta}(x) &\sim& \sum_{j\ge 1} d_{j\alpha\beta} (-x)^{-j\alpha}\\
f_{\alpha \beta}(x) &\sim& \sum_{j\ge 1}\alpha d_{j\alpha\beta} (-x)^{-1-j\alpha}
\label{asymp-}
\end{eqnarray}
Consider now the "free" equation
\begin{equation} \partial_t u(x,t) + A u(x,t) = 0
 \label{rde_0}
\end{equation}
where the symbol of $A$ is given by \eqref{symbol1}, with $\rho$ given by 
\begin{equation}
\label{rho}
\rho = \begin{cases}\beta (\alpha - 1 + sign(1-\alpha))   \quad (\alpha \ne 1) \\
\beta \quad  (\alpha = 1)\, .
\end{cases}
\end{equation}
\smallskip
\emph{Throughout the rest of the paper, we shall only consider operators $A$ with symbol $p$ in \eqref{symbol1} for which $\rho$ is of the form \eqref{rho} for some $\beta \in (-1,1)$. The relation \eqref{rho} between $\beta$ and $\rho$ will always be assumed.}
Taking the Fourier transform then shows that equation \eqref{rde_0} has the fundamental solution
\[(x,t) \mapsto W(x,t) = t^{-1/\alpha}f_{\alpha \beta}(xt^{-1/\alpha})
\]
with initial data $W(x,0) = \delta_0(x)$, the delta distribution. In particular, $W(x,1) = f_{\alpha\beta}(x)$. There is also the special solution
\[(x,t) \mapsto V(x,t) = F_{\alpha \beta}(xt^{-1/\alpha})
\]
with initial data $V(x,0) = H(x)$, the Heaviside function. The equations hold in the sense of distributions, and the initial data are attained in this sense.

\medskip
From now on,  let $\alpha, \, \beta $ be fixed. For fixed $c \in \R$ and $\tau > 0$ we consider the function
\begin{equation}
U_{\tau}(\xi) = F_{\alpha \beta}\left(\xi\tau^{-1/\alpha} \right) \, . \label{wave_def}
\end{equation}
Set $u_{c\tau}(x,t) = U_{\tau}(x + ct)$, then
\begin{eqnarray*}
A u_{c\tau}(x,t) &=&
- \partial_\tau F_{\alpha \beta}\left((x+ct)\tau^{-1/\alpha} \right) \\
&=&
\frac{1}{\alpha} \left((x+ct)\tau^{-1/\alpha-1} \right) f_{\alpha \beta}\left((x+ct)\tau^{-1/\alpha} \right)
\end{eqnarray*}
and
\[
\partial_t u_{c\tau}(x,t) =
c \tau^{-1/\alpha}f_{\alpha \beta}\left((x+ct)\tau^{-1/\alpha} \right)
\]
and therefore 
\begin{eqnarray*}
\partial_t u_{c \tau}(x,t) + A u_{c \tau} (x,t) &=&
\left(\frac{1}{\alpha} \left((x+ct)\tau^{-1/\alpha-1} \right) + c \tau^{-1/\alpha} \right) \times \dots \\
&\quad& \dots \times f_{\alpha \beta}\left((x+ct)\tau^{-1/\alpha} \right)
\end{eqnarray*}
Now $(x+ct)\tau^{-1/\alpha} = F_{\alpha \beta}^{-1}(u_{c \tau}(x,t))$ and consequently
\begin{equation} \label{wave_eq}
\partial_t u_{c \tau}(x,t) + A u_{c \tau} (x,t) = c \tau^{-1/\alpha} g_0(u_{c \tau}(x,t)) +
\frac{1}{\alpha\tau} g_1(u_{c \tau}(x,t))
\end{equation}
with
\begin{equation} \label{g_def}
g_0(\zeta) =   f_{\alpha \beta}(F_{\alpha \beta}^{-1}(\zeta)), \quad
g_1(\zeta) = F_{\alpha \beta}^{-1}(\zeta) f_{\alpha \beta}(F_{\alpha \beta}^{-1}(\zeta)) \, .
\end{equation}
Equation \eqref{wave_eq} is of the form \eqref{rde_alpha}, with $g(\zeta) =  c \tau^{-1/\alpha} g_0(\zeta) + \frac{1}{\alpha\tau} g_1(\zeta)$.

\smallskip
In the case $\alpha = 1$ and $-1 < \beta < 1$, everything is explicit.  Let $\kappa = \cos \frac{\pi \beta}{2}$ and $\sigma = \sin \frac{\pi \beta}{2}$. Then by results in \cite{zolo},
\begin{eqnarray*}
F(x) &=& \frac{1}{2} + \frac{1}{\pi} \arctan \frac{x - \sigma}{\kappa} \\
F^{-1}(\zeta) &=& \sigma -\kappa \cot (\pi \zeta) \\
f(x) &=& \frac{1}{\pi \kappa (1+(x-\sigma)^2/\kappa^2)} \\
g_0(\zeta) &=& \frac{1}{\kappa \pi} \sin^2(\pi \zeta) \\
g_1(\zeta) &=& \frac{\sigma}{\pi\kappa} \sin^2 (\pi \zeta) -\frac{1}{\pi} \cos (\pi \zeta) \sin (\pi \zeta) 
\end{eqnarray*}

It remains to characterize the functions $g_0, \, g_1$ in the general case.

\begin{prpstn} \label{prop2} Let $0 < \alpha < 2, \, -1 < \beta < 1$. The functions $g_0, \, g_1$ have the following properties.

a) $g_0$ and $g_1$ are infinitely differentiable on $(0,1)$.

b) The function $g_0$ is positive on $(0,1)$. The function $g_1$ is negative on $(0,F_{\alpha\beta}(0))$ and positive on $(F_{\alpha\beta}(0),1)$. The function  
\[\zeta \mapsto c \tau^{-1/\alpha} g_0(\zeta) + \frac{1}{\alpha\tau} g_1(\zeta)
\]
is negative on $(0,u^*)$ and positive on $(u^*,1)$, where $u^* = F_{\alpha \beta}(-c \alpha \tau^{-1/\alpha + 1})$.

c) As $\zeta \downarrow 0$, $g_0(\zeta) = O(\zeta^{1 + 1/\alpha})$ and  $g_0(1-\zeta) = O(\zeta^{1 + 1/\alpha})$.

d) As $\zeta \downarrow 0$, $g_1(\zeta) = - \alpha \zeta + O(\zeta^{1 + 1/\alpha})$. As $\zeta \uparrow 1$,  $g_1(\zeta) = \alpha (1-\zeta) + O((1-\zeta)^{1 + 1/\alpha})$. 

e) The functions $g_0$ and $g_1$ can be represented as
\begin{eqnarray*}
g_0(\zeta) &=& \frac{d}{d \zeta} \int_{-\infty}^{F_{\alpha \beta}^{-1}(\zeta)} f_{\alpha \beta}^2(s) ds \\
g_1(\zeta) &=& \frac{d}{d \zeta} \int_{-\infty}^{F_{\alpha \beta}^{-1}(\zeta)} sf_{\alpha \beta}^2(s) ds
 \, .
\end{eqnarray*}

\end{prpstn}
\begin{proof} Property a) follows since $f_{\alpha\beta}$ and $F{\alpha \beta}$ together with its inverse are infinitely differentiable. Property b) is obvious. Properties c) and d) follow from the asymptotic expansions \eqref{asymp+} and \eqref{asymp-}. Finally e) can be checked by differentiation.
\end{proof}
Property e) will not be used in what follows.  It should be noted that $g_0$ and $g_1$ are of class $C^1$ on $[0,1]$, but are not infinitely differentiable at the interval endpoints, except if $\alpha = 1$. Clearly $g_0$ and $g_1$ do not depend on $c$ or $\tau$. We are therefore free to form fairly arbitrary linear combinations of $g_0$ and $g_1$ by choosing $c$ and $\tau$. 

\smallskip
The construction provides travelling wave solutions for \eqref{rde_alpha} for a special class of functions for which $g \in C^1([0,1])$, $g(0) = g(u^*) = g(1)$ for some $u^* \in (0,1)$, and $g'(0) < 0, \, g'(u^*) > 0, \, g'(1) < 0$. This raises the possibility that \eqref{rde_alpha} possesses travelling wave solutions for more general functions $g$ with these properties. 

\smallskip
If the same construction is attempted for the case $\alpha = 2$, it turns out that $g_0$ and $g_1$ are merely continuous on $[0,1]$, with derivatives that have logarithmic singularities near $\zeta = 0$ and $\zeta = 1$. Therefore the arguments in the next section cannot be extended to the case $\alpha = 2$, and indeed the results of the next section do not hold in that case. 

\section{Results on the Speed of Spread}

This section contains the main results of this paper. As before, the operator $A$ has symbol \eqref{symbol1} with $0 < \alpha < 2$ and $\rho$ satisfying \eqref{rho} with $-1 < \beta < 1$.  We always assume that $u$ is a solution of \eqref{rde_alpha} and that $g \in C^1([0,1],\R)$ with $g(0) = g(1) = 0$. Initial data $u_0$ will be assumed to satisfy
\begin{equation}
\label{initial}
u_0 \in C(\R,\R), \, 0 \le u_0(x) \le 1, \, \lim_{x \to \infty} u_0(x) = 1, \, supp(u_0) \subset [0,\infty)\, .
\end{equation}
The results in section 4 then imply that  \eqref{rde_alpha} has a unique mild solution $u$ that exists for all $x \in \R, \, t>0$, and this solution satisfies  $0 \le u(x,t) \le 1$ for all $(x,t)$. The notation of that section will also be used here.

\smallskip
We first discuss the case where $g \ge 0$ on $(0,1)$. The main result in this case is the following. 

\begin{thrm} Let $u$ be the solution of \eqref{rde_alpha} with $u_0$ satisfying \eqref{initial}. 

a) Let $\alpha > 1$. Assume that $g>0$ on $(0,1)$ and that for some $c_0>0, \, 0 < \gamma < \frac{\alpha}{\alpha - 1}$ and all $\zeta \in [0,\frac12]$
\[g(\zeta) \ge c_0 \zeta^\gamma \, .
\]
Then for all $c> 0$
\[ \liminf_{t \to \infty} \inf_{x \ge -ct} u(x,t) = 1 \, .
\]

b) Let $\alpha = 1$. Assume that $g>0$ on $(0,1)$. Then for all $c> 0$
\[ \liminf_{t \to \infty} \inf_{x \ge -ct} u(x,t) = 1 \, .
\]

b) Let $\alpha < 1$. Assume that $g\ge 0$ on $(0,1)$ and $g(\zeta) > 0$ for $ \zeta \in [\frac{1-\beta}{2},1)$. Then for all $c> 0$
\[ \liminf_{t \to \infty} \inf_{x \ge -ct} u(x,t) = 1 \, .
\] 
\end{thrm}

The result shows that the speed of spread is unbounded (that is, \eqref{est_1} holds for all $c > 0$), and it exhibits different mechanisms for this phenomenon. Recall that in the interpretation of \cite{kpp}, the function $g$ is responsible for the growth of a substance whose density is given by $u$, while $A$ describes the spread of this substance.  If $\alpha \in (0,2)$, the substance spreads with a jump process, not with Brownian diffusion, and jumps of magnitude $h$ occur with a probability that is $O(h^{-\alpha})$ for large $h$. If $\alpha > 1$, the mean jump distance is still finite. In this case, the growth rate $g(u)$ at small densities (small $u$) is responsible for the unbounded speed of spread. If $\alpha$ is close to 1, this growth can be very weak ($g(\zeta) \sim \zeta^\gamma$ with large $\gamma$), yet the speed of spread is still unbounded. If on the other hand $\alpha < 1$, i.e. jump sizes have unbounded mean, the growth rate for small densities does not matter any more for the speed of spread to be unbounded; in fact there may be no growth at all for small densities ($g(u) = 0$ for small $u$), and yet the speed of spread is unbounded. In this case, the unbounded speed of spread results from growth that occurs solely for large densities ($g(\zeta) > 0$ only for $\zeta \ge \frac{1-\beta}{2}$). The substance is transported towards $-\infty$ due large ($\alpha < 1$) negative jumps, resulting in an unbounded speed of spread.  It is known that in this case, $\frac{1-\beta}{2}$ is the fraction of negative jumps. If this fraction is large, then it is sufficient that growth occurs only for densities close to the maximal value, i.e. $g(u) > 0$ on $[\frac{1-\beta}{2},1)$ already implies that the speed of spread is unbounded.  The case $\alpha = 1$ is intermediate: Any growth for small densities ($g(\zeta) > 0$ for $\zeta > 0$) results in an unbounded speed of spread. 

\smallskip
In the case $\alpha > 1$, it would be interesting to know if the speed of spread is still unbounded if $\gamma \ge \frac{\alpha}{\alpha - 1}$ or if a finite speed of spread occurs (\eqref{est_0} holds for large $c$) if $\gamma$ becomes sufficiently large, i.e. if growth is extremely weak for small densities $u$. In the case $\alpha < 1$, it would be interesting to know if a finite speed of spread is possible at all if $g\ge0$ and $g$ is not identically equal to 0. 

\smallskip
The main result in the case where $g$ is negative near $\zeta = 0$ is the following.

\begin{thrm} Let $u$ be the solution of \eqref{rde_alpha} with initial data $u_0$ satisfying \eqref{initial}. Assume that $g'(0)< 0$. Then there exists $c> 0$ such that
\[ \limsup_{t \to \infty} \sup_{x \le -ct} u(x,t) = 0 \, .
\]
\end{thrm}
The result shows that negative proportional growth at small densities ($g'(0) < 0$) always limits the speed of spread of a substance whose growth and spread are governed by \eqref{rde_alpha}, even for processes whose jump sizes tend to be very large ($\alpha < 1$). I am not aware of an interpretation of \eqref{rde_alpha} in the context of material science, similar to the use of \eqref{rde2} in \cite{allencahn}. 

\smallskip
The proofs will be given below.  The main tools in the proofs are the comparison arguments given in the next section, together with the following auxiliary result.

\begin{lmm} \label{lmm_comp} Let $g \in C^1([0,1],\R)$ and let $g_0, \, g_1$  be defined as in \eqref{g_def}, depending on $\alpha \in (0,2), \, \beta \in (-1,1)$.

a) Let $\alpha \in (1,2)$. Suppose that $g(\zeta) > 0$ for all $\zeta \in (0,1]$ and that there exist $c_0>0$ and $\gamma < \frac{\alpha}{\alpha - 1}$ such that $g(\zeta) \ge c_0\zeta^\gamma$ for all $\zeta \in [0,\frac12]$. Then given any $c>0$ there exists $\tau > 0$ such that for all $\zeta \in [0,1]$
\[g(\zeta) \ge c\tau^{-1/\alpha}g_0(\zeta) + (\alpha \tau)^{-1} g_1(\zeta) \, .
\] 

b) Let $\alpha = 1$. Suppose that $g > 0$ on $(0,1]$. Then given any $c>0$ there exists $\tau > 0$  such that for all $\zeta \in [0,1]$
\[g(\zeta) \ge c\tau^{-1}g_0(\zeta) + \tau^{-1} g_1(\zeta) \, .
\]

c) Let $\alpha \in (0,1)$. Suppose that $g\ge 0$ on $[0,1]$ and $g(\zeta) > 0$ for all $\zeta \in [\frac{1-\beta}{2},1)$. Then given any $c>0$ there exists $\tau > 0$  such that for all $\zeta \in [0,1]$
\[g(\zeta) \ge c\tau^{-1/\alpha}g_0(\zeta) + (\alpha \tau)^{-1} g_1(\zeta) \, .
\]

d) Suppose $g'(0) < 0$ and $g(\zeta) = 0$ for $\zeta \in [1-\epsilon,1]$ for some $\epsilon$. Then there exist $c \in \R$ and $\tau > 0$  such that for all $\zeta \in [0,1]$
\[g(\zeta) \le c\tau^{-1/\alpha}g_0(\zeta) + (\alpha \tau)^{-1} g_1(\zeta) \, .
\] 
\end{lmm}

\begin{proof}  Consider first statement a). Let $\alpha >1, \, -1 < \beta < 1$ and let $M>0$ be large enough such that for some $c_1, c_2 > 0$ and all $x \le -M$
\[F_{\alpha \beta}(x) \ge c_1 |x|^{-\alpha}, \quad f_{\alpha\beta}(x) \le c_2 |x|^{-1-\alpha}\, .
\]
This is possible by \eqref{asymp-}. Let $c>0$ be given, then we may increase $M$ further such that also 
\[c_0c_1^\gamma M^r \ge c_2 \frac{\alpha - 1}{\alpha} c^{\frac{\alpha}{\alpha-1}} \, .
\]
where $r = \frac{\alpha}{\alpha-1} - \gamma > 0$. We omit the subscript $\alpha\beta$ in the formulae involving $F_{\alpha\beta}$ and $f_{\alpha\beta}$ from now on.  Now set $\delta = F(-M)$. Then for $0 < \zeta = F(x)  \le \delta$, i.e. $x < -M$, and for all $\tau \ge \left( M/c
\right)^{\alpha/(\alpha-1)}$
\begin{eqnarray*}
c\tau^{-1/\alpha} g_0(F(x)) + (\alpha \tau)^{-1}g_1(F(x)) &=& \left(c\tau^{-1/\alpha} + (\alpha \tau)^{-1} x \right) f(x) \\
&\le& \frac{\alpha-1}{\alpha} c^{\frac{\alpha}{\alpha-1}}|x|^{1/(1-\alpha)} f(x)\\
&\le&  \frac{\alpha-1}{\alpha} c^{\frac{\alpha}{\alpha-1}}|x|^{1/(1-\alpha)} c_2|x|^{-1-\alpha}\\
 &=&
\frac{\alpha-1}{\alpha} c^{\frac{\alpha}{\alpha-1}}c_2|x|^{-\alpha^2/(\alpha-1)}
\end{eqnarray*}  
where a standard calculus argument has been used to see that the expression  
$\left(c\tau^{-1/\alpha} + (\alpha \tau)^{-1} x \right) $ is maximal for $\tau = \left(|x|/c \right)^{\alpha/(\alpha-1)}$. We estimate further, using the choice of $M$  
\begin{eqnarray*}
c\tau^{-1/\alpha} g_0(F(x)) + (\alpha \tau)^{-1}g_1(F(x)) &\le& c_0 c_1^\gamma M^r |x|^{-\alpha^2/(\alpha-1)} \\
&\le& c_0 c_1^\gamma |x|^r |x|^{-\alpha^2/(\alpha-1)} = c_0 \left(c_1 |x|^{\alpha} \right)^\gamma \\
&\le& c_0\left(F(x)\right)^\gamma \le g(F(x)) \, .
\end{eqnarray*}
Therefore, for all $\tau \ge \left( M/c \right)^{\alpha/(\alpha-1)}$ and all $\zeta < \delta = F(-M)$,
\[c\tau^{-1/\alpha} g_0(\zeta)) + (\alpha \tau)^{-1}g_1(\zeta) \le g(\zeta) \, .
\]
Since $g>0$ on $[\delta,1]$ by assumption, this inequality can be achieved also on $[\delta,1]$ by increasing $\tau$ even further. This proves part a).

\smallskip
The proof of part b) is straight forward: Given $c>0$, note that $cg_0(\zeta) + g_1(\zeta) \le 0$ on the interval $[0, F(-c)]$. Then $\dfrac{cg_0(\zeta) + g_1(\zeta)}{\tau} \le g(\zeta)$ is true if $\tau$ is sufficiently large. 

\smallskip
To prove part c), let again $c>0$ be given. Let $u^* = \inf \{u \in [0,1] \, | g(u) > 0 \}$. Then $u^* < \frac{1-\beta}{2} = F(0)$. Pick $\tau$ large enough such that 
$F\left(-c\alpha \tau^{1-1/\alpha}\right) > u^*$. This is possible since $\alpha < 1$. Then on $[0,u^*]$, 
\[c\tau^{-1/\alpha} g_0(\zeta)) + (\alpha \tau)^{-1}g_1(\zeta) \le g(\zeta) 
\]
since the left hand side is non-positive there by Proposition \ref{prop2}. By increasing $\tau$ further, we can obtain this estimate also for $\zeta \in [u^*,1]$, using again that $g$ is assumed to be positive on $[\frac{1-\beta}{2},1]$.  

\smallskip
To prove part d), note first that 
\[c\tau^{-1/\alpha} g_0(\zeta)) + (\alpha \tau)^{-1}g_1(\zeta) \ge  g(\zeta) 
\]
on an interval $[0,\delta]$ as soon as $\tau^{-1} + g'(0) > 0$, i.e. for sufficiently small $\tau$. Increasing $c$ sufficiently and noting that $g=0$ near $\zeta = 1$ extends this inequality to the entire interval $[0,1]$. 
\end{proof}

\noindent
\emph{Proof of Theorem 3.1.} The proof uses the same argument for all three parts, so we give details only in part a). Let $\epsilon > 0$. We replace $u$ with $\tilde u = (1+\epsilon)u$ and $g$ with $\tilde g$, where $\tilde g(\zeta) = (1+\epsilon)g((1+\epsilon)^{-1}\zeta)$. Then 
\[\partial_t \tilde u + A \tilde u = \tilde g(\tilde u) \, .
\]
Then $\tilde g(\zeta) \ge \tilde c_0 \zeta^\gamma$ for $\zeta \in [0,\frac12]$, possibly with a changed $c_0$, and additionally $\tilde g>0$ on $(0,1]$. Let $c>0$ be given, then there exists $\tau >0$ such that 
\begin{equation} 
\tilde g(\zeta) \ge (c+1)\tau^{-1/\alpha}g_0(\zeta) + (\alpha \tau)^{-1} g_1(\zeta) \, .
\label{eqn_c}
\end{equation} 
for all $\zeta \in [0,1]$ by Lemma \ref{lmm_comp}. By extending $g_0$ and $g_1$ to be zero on $[1,1+\epsilon]$, this inequality is true on $[0,1+\epsilon]$. Now find a constant $d$ such that $v_0(x) = H(x-d) \le \tilde u(x,0)$ for all $x$. This is possible since $\lim_{x\to \infty} \tilde u(x,0) = 1+\epsilon$. By Proposition \ref{comparison_step}, we see that 
\[\tilde u(x,t) \ge F\left( (x-d)t^{-1/\alpha} \right)
\]  
for all $x \in \R, \, t>0$. This is in particular true for $t = \tau$. Now use Proposition \ref{comparison} and \eqref{eqn_c} to infer that
\[\tilde u(x,t) \ge F\left( (x-d+(c+1)t)\tau^{-1/\alpha} \right)
\] 
for all $x \in \R, t \ge \tau$. Therefore for $t \ge \tau$ and $x \ge -ct$, 
\begin{eqnarray*}
\tilde u(x,t) &\ge& F\left( (x-d+(c+1)t)\tau^{-1/\alpha} \right) \\
&\ge& F\left( (-ct-d+(c+1)t)\tau^{-1/\alpha} \right) \\ 
&=&  F\left( (t-d)\tau^{-1/\alpha} \right) \, .
\end{eqnarray*}

As $t \to \infty$, the right hand side goes to 1. Rewriting this in terms of $u$, we see that 
\[
\liminf_{t \to \infty} \inf_{x \ge -ct} \tilde u(x,t) \ge (1+\epsilon)^{-1} \,. \]
Since $\epsilon$ was arbitrary, the desired result follows.

\smallskip
In case of part b), the same argument can be used without changes, appealing to part b) of Lemma \ref{lmm_comp}. 

\smallskip
In case of part c), we have to restrict $\epsilon$ such that $\tilde g > 0$ on $[\frac{1-\beta}{2},1]$, that is, $g > 0$ on $[\frac{1-\beta}{2(1+\epsilon)},1]$. The rest of the proof is again unchanged, using part c) of Lemma \ref{lmm_comp}. \hspace*{3.9in} $\square$ 

\bigskip

\noindent
\emph{Proof of Theorem 3.2.} We replace $u$ with $\tilde u = \frac12u$ and $g$ with $\tilde g$, where $\tilde g(\zeta) = \frac12g(2\zeta)$ for $0 \le \zeta \le \frac12$ and $\tilde g(\zeta) = 0$ for $\zeta \in (\frac12,1]$. Then $\tilde g'(0) = g'(0) < 0$ and 
\[\partial_t \tilde u + A \tilde u = \tilde g(\tilde u) \, .
\]
By Lemma \ref{lmm_comp}, part d), there exist $c > 0$ and $\tau > 0$ such that 
\[
\tilde g(\zeta) \le (c-1)\tau^{-1/\alpha}g_0(\zeta) + (\alpha \tau)^{-1} g_1(\zeta) \, .
\]
Since $\lim_{x \to \infty}\tilde u(x,0) = \frac12$ and $u(x,0) = 0$ for $x<0$, we can find $d>0$ such that $F\left(x+d)\tau^{-1/\alpha} \right) \ge \tilde u(x,0)$
for all $x \in \R$. Using Proposition \ref{comparison}, one sees that
\[ F\left( (x+d+(c-1)t)\tau^{-1/\alpha} \right) \ge \tilde u(x,t)
\] 
for all $x \in \R, t >0$. Therefore for $t >0$ and $x \le -ct$, 
\begin{eqnarray*}
\tilde u(x,t) &\le& F\left( (x+d+(c-1)t)\tau^{-1/\alpha} \right) \\
&\le& F\left( (-ct+d+(c-1)t)\tau^{-1/\alpha} \right) \\ 
&=&  F\left( (-t+d)\tau^{-1/\alpha} \right) \, .
\end{eqnarray*} 
The right hand side tends to $0$ as $t \to \infty$. In terms of $u$, this implies
\[
\limsup_{t \to \infty} \sup_{x \le -ct}  u(x,t) = 0 \,. \]
This concludes the proof. \hspace*{3.in} $\square$ 

\section{Facts About Fractional Reaction-Diffusion Equations}

In this section we summarize some basic theory about \eqref{rde_alpha} that is needed in this note. A broader and deeper discussion may be found in \cite{bkm2}.

\smallskip
We work in the Banach space
\[\cC_{lim} = \{w \in C(\R) \, | \, \lim_{x \to \infty} w(x) \quad \text{and} \quad
\lim_{x \to - \infty} w(x) \quad \text{exist} \}
\]
equipped with the supremum norm $\| \cdot \|$.
Let $A$ be the pseudodifferential operator with symbol \eqref{symbol1} and parameters $\alpha, \, \rho$. As always, let $\rho, \, \beta$ be related by \eqref{rho} and $-1 < \beta < 1$. Solutions of the free equation \eqref{rde_0} with initial data $u(\cdot,0) = \varphi \in \cC_{lim}$ then can be written in terms of the fundamental solution $(x,t) \mapsto t^{-1/\alpha} f_{\alpha \beta}\left( \frac{x-y}{t^{1/\alpha}} \right) $, namely 
\begin{equation}
u(x,t) = \int_\R t^{-1/\alpha} f_{\alpha \beta}\left( \frac{x-y}{t^{1/\alpha}} \right) \varphi(y) \, dy
\label{def_S}
\end{equation}
where $f_{\alpha \beta}$ is a stable probability density function. 

\smallskip
For fixed $\alpha$ and $\rho$ and $\varphi \in \cC_{lim}$, define
\[S(t)\varphi(x) = u(x,t)\]
where $u$ is given by  \eqref{def_S}. This is a positive $C_0$ semigroup on $\cC_{lim}$ and a Feller semigroup on the subspace of functions in $\cC_{lim}$ that vanish at $\pm \infty$. If $\psi$ is a continuous function from $[0,T]$ to $\cC_{lim}$, then solutions of the inhomogeneous equation
\begin{equation}
\label{inhomo}
\partial_t u(x,t)  + A u(x,t) = \psi(x,t), \quad u(\cdot,0) = \varphi
\end{equation}
can be written with the variation-of-constants formula
\begin{equation}
\label{var_const}
u(\cdot,t) = S(t)\varphi + \int_0^t S(t-s) \psi(\cdot,s) \, ds \, .
\end{equation}
A continuous curve $u:[0,T] \to \cC_{lim}$ that satisfies \eqref{var_const} is commonly called a mild solution of \eqref{inhomo}.
Next let $g:[0,\infty) \times \R \to \R$ be locally Lipschitz continuous in both variables and let $\varphi \in \cC_{lim}$. Then the equation $\partial_t u(x,t) + A u(x,t) = g(t,u(x,t))$ (for which \eqref{rde_alpha} is a special case) has a unique mild solution $u \in C\left([0,T),\cC_{lim}\right)$, where $0 < T \le \infty$ is maximal. Either $T= \infty$, or $\|u(\cdot,t)\| \to \infty$ as $t \uparrow T$. The solution can be obtained as the locally in time uniform limit of the iteration scheme
\[u_{n+1}(\cdot,t) = S(t)\varphi + \int_0^t S(t-s) g(s,u_n(\cdot,s)) \, ds \quad (n = 0, \, 1, \dots)
\]
with $u_0$ arbitrary, e.g. $u_0(\cdot,t) = S(t) \varphi$.
It is possible to set up a more general solution theory, but this is not needed for the purposes of this paper.

\smallskip
Solutions of \eqref{rde_alpha} satisfy comparison theorems. Results of this type are true for all Feller semigroup. A systematic study of such semigroups and their generators was carried out in \cite{bcp}, following the seminal work on this topic in \cite{waldi}. For the sake of completeness, a comparison result is stated here, and its proof is sketched.

\begin{prpstn}
\label{comparison}
Let $u, \, v \in C\left([0,T], \cC_{lim} \right)$ be mild solutions of the equations
\[\partial_t u + Au = g(u), \quad \partial_t v + Av = h(v)
\]
where $g, \, h :\R \to \R$ are locally Lipschitz continuous. If
\[g(\zeta) \le h(\zeta) \quad \forall \zeta \in \R
\]
and
\[u(\cdot,0) \le v(\cdot,0) \]
then
\[u(x,t) \le v(x,t) \quad \forall (x,t) \in \R \times [0,T] \, .
\]
\end{prpstn}

\begin{proof}
Let $M = \max_{[0,T]} \left( \|u(\cdot,t)\| + \|v(\cdot,t)\| + 1 \right)$. Let $\lambda > |g'(\zeta)| + h'(\zeta)|$ for all $|\zeta| \le M$. Without loss of generality we may assume that $g$ and $h$ are constant outside $[-M,M]$. Set
\[U(x,t) = e^{\lambda t} u(x,t), \quad V(x,t) = e^{\lambda t}v(x,t)
\]
and observe that $U$ and $V$ satisfy
\begin{eqnarray*}
\partial_t U + AU &=& \tilde g(t,U) \\
\partial_t V  + AV &=& \tilde h(t,V)
\end{eqnarray*}
with $\tilde g(t,\zeta) = \lambda \zeta + e^{\lambda t} g\left( e^{-\lambda t} \zeta \right)$ and $\tilde h(t, \zeta)$ defined similarly. Clearly, $\tilde g(t,\zeta) \le \tilde h(t,\zeta)$ for all $\zeta$. The function $\tilde g$ is non-decreasing in its second argument, since for almost all $\zeta$
\[
\partial_\zeta \tilde g(t,\zeta) = \lambda + g'(e^{-\lambda \zeta}) \ge 0 \, .
\]
Consider the iteration scheme
\[
U_{n+1}(\cdot,t) = S(t)u(\cdot,0) + \int_0^t S(t-s) \tilde g(s,U_n(\cdot,s)) \, ds
\]
and similarly for $V_n$ and $\tilde h$. The scheme for the $U_n$ converges to the limit $U$, and the scheme for the $V_n$ converges to the limit $V$.

\smallskip
We now employ a standard induction argument to show that $U_n \le V_n$ on $\R \times [0,T]$ for all $n$. This implies that $U \le V$ and therefore also $u \le v$ on $\R \times [0,T]$. Let $U_0(\cdot,t) = S(t)u(\cdot,0)$ and $V_0(\cdot,t) = S(t)v(\cdot,0)$, then $U_0 \le V_0$ on $\R \times [0,T]$ since $S$ is a positive semigroup and $u(\cdot,0) \le v(\cdot,0)$.  Suppose $U_n \le V_n$ on $\R \times [0,T]$, then
\begin{eqnarray*}
U_{n+1}(\cdot,t) &=& S(t)u(\cdot,0) + \int_0^t S(t-s) \tilde g(s,U_n(\cdot,s)) ds \\
&\le& S(t)v(\cdot,0) + \int_0^t S(t-s) \tilde g(s,U_n(\cdot,s)) ds \\
&\le& S(t)v(\cdot,0) + \int_0^t S(t-s) \tilde g(s,V_n(\cdot,s)) ds \\
&\le& S(t)v(\cdot,0) + \int_0^t S(t-s) \tilde h(s,V_n(\cdot,s)) ds \\
&=& V_{n+1}(\cdot,t)
\end{eqnarray*}
which completes the induction step. This proves the proposition.
\end{proof}

\begin{crllr}
Consider a mild solution $u \in C\left([0,T),\cC_{lim} \right)$ of \eqref{rde_alpha} and assume that $g$ is locally Lipschitz continuous. If $g(\gamma) \ge 0$ for some $\gamma$ and $u(\cdot,0) \ge \gamma$, then $u(\cdot,t) \ge \gamma$ for all $t$. If $g(\gamma) \ge 0 \ge g(\delta) $ for some $\gamma<\delta$ and $\gamma \le u(\cdot,0) \le \delta$, then $\gamma \le u(\cdot,t) \le \delta$ for all $t$, and the solution can be continued to $\R \times [0,\infty)$.
\end{crllr}
The proof consists in observing that the constant functions $v(x,t) = \gamma$ and $w(x,t) = \delta$ solve \eqref{rde_alpha} with right hand sides $0$ and therefore must be pointwise bounds for the solution, by Proposition \ref{comparison}. If the solution remains bounded between two constants, then its supremum norm remains bounded and it can be continued to $\R\times [0,\infty)$.

\smallskip
Also required is a comparison result for solutions whose initial data are step functions. Since such initial data are not in $\cC_{lim}$, a separate argument is required.

\begin{prpstn}
\label{comparison_step}
Let $u \in C\left([0,T],\cC_{lim}\right)$ be a mild solution of \eqref{rde_alpha}, with locally Lipschitz continuous $g$. Assume that
\[u(x,0) \ge v_0(x) = a_0 +\sum_{j=1}^N a_j H(x-c_j) \quad \forall x \in \R
\]
where $a_j \in \R, \, c_1 < c_2 \dots < c_N$, and $H$ is the Heaviside function. Let $\gamma = \min_\R v_0(x)$ and $\delta = \max_\R v_0(x)$. Assume also that $g \ge 0$ on $[\gamma,\delta]$. Then
\begin{equation} \label{est_2}
u(x,t) \ge a_0 + \sum_{j=1}^N a_j F_{\alpha \beta}\left( \frac{x-c_j}{t^{1/\alpha}} \right) \quad \forall x \in \R, 0 < t \le T
\end{equation}
where $F_{\alpha\beta}$ is the cumulative distribution function of the associated stable distribution.
\end{prpstn}

\begin{proof}
We know that $u(x,t) \ge \gamma$ and thus may assume that $g(\zeta) \ge 0$ also for $\zeta < \gamma$. For arbitrary $\epsilon \sigma > 0$, we set
\[
v_{\epsilon\sigma}(x) = \frac{1}{\sigma} \int_0^\sigma v_0(x-z) dz - \epsilon \, .
\]
Then $v_{\epsilon \sigma}$ is piecewise linear and constant outside the interval $[c_1,c_N+\sigma]$; in particular, $v_{\epsilon \sigma} \in \cC_{lim}$. Solving \eqref{rde_0} with initial data $v_{\epsilon,\sigma}$ gives the solution
\[V_{\epsilon \sigma}(x,t) = \frac{1}{\sigma} \int_0^\sigma
\sum_{j=1}^N a_j F_{\alpha \beta}\left( \frac{x-z-c_j}{t^{1/\alpha}} \right) +a_0 - \epsilon \, .
\]
Given $\epsilon > 0$, it is possible to find $\sigma > 0$ such that $v_{\epsilon \sigma} < u_0(x)$ on $\R$, since $u_0$ is uniformly continuous. Clearly, $\gamma-\epsilon \le V_{\epsilon\sigma} \le \delta$. We may therefore view $V_{\epsilon \sigma}$ as a solution of \eqref{rde_alpha} with a right hand side $h$ that satisfies $h(\zeta) = 0\le g(zeta)$ for $\zeta \le \delta$ and $h(\zeta) \le g(\zeta)$ also for $\zeta > \delta)$.  Then by Proposition \ref{comparison}
\[u(x,t) \ge V_{\epsilon, \sigma}(x,t) \quad \forall x \in \R, \, 0 < t \le T \, .
\]
Send $\delta$ to $0$, then since $F_{\alpha \beta}$ is uniformly continuous, \eqref{est_2} is obtained with $a_0$ replaced by $a_0 - \epsilon$ on the right hand side. Now send $\epsilon$ to $0$ and \eqref{est_2} follows.
\end{proof}


\begin{thebibliography}{99}

\bibitem{allencahn} S.M. Allen, J.W. Cahn, A microscope theory for antiphase boundary motion and its application to antiphase domain coarsening. Acta Metall. {\bf 27} no. 6 (1979),  1085 -- 1095.

\bibitem{aronwein}
D.G. Aronson, H.F. Weinberger, Multidimensional nonlinear diffusion arising in population genetics.
Adv. Math. \textbf{30} (1978), pp. 33-76.


\bibitem{bkm2}B. Baeumer, M. Kov\'acs, M. M. Meerschaert,
Fractional reproduction-dispersal equations and heavy tail dispersal kernels. Bull. Math. Biology {\bf 69} No.7  (2007), 2281 -- 2297.

\bibitem{bkm3}B. Baeumer, M. Kov\'acs, M. M. Meerschaert,
Numerical solutions for  fractional reaction-diffusion equations.
Computers \& Mathematics with Applications {\bf  55} No. 10 (2008), 2212 -- 2226.

\bibitem{bcp}
J.-M. Bony, P. Courr\`ege, and P. Priouret,
Semi-groupes de Feller sur une vari\'et\'e \`a bord compacte et probl\`emes aux limites int\'egro-diff\'erentiels du second ordre donnant lieu au principe du maximum.  Ann. Inst. Fourier {\bf 18}  (1968), pp. 369--521.

\bibitem{cabre1} X. Cabr\'e, J.-M, Roquejoffre, Propagation de fronts dans les \'equations de
 Fisher-KPP avec diffusion fractionnaire. 2009. arxiv: 0905.1299[math]. 9 pp.

\bibitem{castil}D. del-Castillo-Negrete, B. A. Carreras, and V. E. Lynch, Front Dynamics in Reaction-Diffusion Systems with Levy Flights: A Fractional Diffusion Approach. Phys. Review Letters  {\bf 91} (2003),


\bibitem{ebert1} U. Ebert, W. v. Saarloos,
Front propagation into unstable states: universal algebraic convergence towards uniformly translating pulled fronts. Physica D {\bf 146} (2000), 1 -- 99.

\bibitem{fortpujol}J. Fort, T. Pujol, Progress in front propagation research. Rep. Prog. Phys. {\bf 71} (2008), 41pp.

\bibitem{henry} B.I. Henry, S.L. Wearne, Fractional reaction-diffusion. Physica A {\bf 276} (2000), 448 -- 455.



\bibitem{kpp}
A.N. Kolmogorov, I.G. Petrovskii, and  N.S. Piskunov,
\'Etude de l'\'equation de diffusion avec accroissement de la quantit\'e de
    mati\`ere, et son application \`a un probl\`eme biologique,
    Bul. Moskowskogo Gos. Univ. \textbf{17} (1937), pp. 1--26.

\bibitem{mvv1}
R. Mancinelli, D. Vergni, and A. Vulpiani,  Superfast front propagation in reactive systems with non-Gaussian diffusion. Europhys. Lett. {\bf 60} (2002), 532 -- 538.



\bibitem{mvv2}
R. Mancinelli, D. Vergni, and A. Vulpiani,
Front propagation in reactive systems with anomalous diffusion,
Phys. D \textbf{185} (2003), pp. 175-195.

\bibitem{mckean} H.P. McKean, Application of Brownian motion to the equation of Kolmogorov-Petrovskii-Piskunov. Comm. Pure Appl. Math. {\bf 28} No.3 (1975), 323 -- 331.

\bibitem{golovin} Y. Nec, A.A. Nepomnyashchy, A.A. Golovin, Front-type solutions of fractional Allen-Cahn equation.Physica D {\bf 237} (2008) 3237 – 3251.

\bibitem{nolan} J. Nolan, \emph{Stable Distributions - Models for Heavy Tailed Data.} Birkh\"auser Boston {\bf 2010}.  \emph{In progress, Chapter 1 online at academic2.american.edu/$\sim$jpnolan}

\bibitem{schaaf} K.W. Schaaf, Asymptotic behavior and traveling wave solutions for parabolic functional differential equations. Trans. AMS {\bf 302} no. 2 (1987), 587 -- 615.

\bibitem{waldi} W. von Waldenfels, Positive Halbgruppen auf einem n-dimensionalen Torus, Archiv der Math. {\bf 15} (1964), 191 -- 203.

\bibitem{zanette} D. H. Zanette, Wave fronts in bistable reactions with anomalous Levy-flight diffusion. Phys. Review E {\bf 55 (1997}, 1181 -- 1184.

\bibitem{zas} G. M. Zaslavsky, Chaos, fractional
kinetics, and anomalous transport. Physics Reports {\bf 317}
(2000), 461 -- 580.

\bibitem{zolo}
V. M. Zolotarev, \emph{One-dimensional Stable Distributions.}
Amer. Math. Soc., Providence, R. I. (1986), translated from
\emph{Odnomernye ustoichivye raspredelniia}, ``Nauka'', Moscow, (1982).


\end{thebibliography}
\end{document}